\documentclass[12pt,reqno]{article}

\usepackage[usenames]{color}
\usepackage[dvipsnames]{xcolor}

\usepackage{graphicx}
\usepackage{amscd}

\definecolor{webgreen}{rgb}{0,.5,0}
\definecolor{webbrown}{rgb}{.6,0,0}

\usepackage{color}
\usepackage{fullpage}
\usepackage{float}

\usepackage{graphics,amsmath,amssymb}
\usepackage{amsthm}
\usepackage{amsfonts}
\usepackage{latexsym}
\usepackage{epsf}
\usepackage{mathrsfs}
\usepackage{mathtools}
\usepackage[T1]{fontenc}

\usepackage[colorlinks=true,
linkcolor=webgreen,
filecolor=webbrown,
citecolor=webgreen]{hyperref}

\begin{document}

\theoremstyle{plain}
\newtheorem{theorem}{Theorem}
\newtheorem{corollary}[theorem]{Corollary}
\newtheorem{lemma}[theorem]{Lemma}
\newtheorem{proposition}[theorem]{Proposition}

\theoremstyle{definition}
\newtheorem{definition}[theorem]{Definition}
\newtheorem{example}[theorem]{Example}
\newtheorem{conjecture}[theorem]{Conjecture}

\theoremstyle{remark}
\newtheorem{remark}[theorem]{Remark}

\begin{center}
\vskip 1cm{\LARGE\bf Connecting the Stirling numbers and \textit{k}-bonacci sums}

\vskip 1cm
\large

Muhammad Adam Dombrowski\\
Brown University\\
Providence, RI 02912\\
 USA\\
\href{mailto:muhammad_adam_dombrowski@brown.edu}{\tt muhammad\_adam\_dombrowski@brown.edu}\\ 
\end{center}
\vskip0.2in

\begin{abstract}
    This paper proves why the Stirling numbers show up in a experimentally determined formula for the $k$-bonaccis. We develop a  bijection between a previously determined summation formula for $k$-bonaccis and an experimentally determined formula, proven algebraically.
\end{abstract}

\section{Introduction}

The Fibonacci sequence has been studied for centuries. Surprisingly, its namesake, Fibonacci, did not know about these numbers until the 12th century, while it had been discovered in India more than a thousand years beforehand.

The premise of the Fibonacci sequence\footnote{This sequence appears on the OEIS\cite{oeis} at \href{https://oeis.org/A000045}{A000045}. Since there are so many OEIS references in this paper, the main website is cited and each individual sequence under the site is hyperlinked.} is simple. We define $F_0$ as 0 and $F_1$ as 1. From there the sequence states that $F_2$ is the sum of $F_0$ and $F_1$, $F_3$ is the sum of $F_1$ and $F_2$, and so forth. Each number is the sum of the two terms before it, and this sequence continues to infinity. 

So, starting with $F_0 = 0$ and $F_1 = 1$, we get the sequence:

\[0,1,1,2,3,5,8,13,21,34,55,89...\]
 
This recursive pattern is quite beautiful and can be found in nature numerous times over, from the spirals on a pine cone to the petals of a flower.

Adding up the two previous numbers to get the subsequent term may become ``boring" after a while. What if instead of adding the two previous numbers, the previous three numbers are added, or previous four numbers, or previous $n$ numbers? These variable sequences are defined as $k$-generalized Fibonacci sequences, where $k$ is the numbers of previous terms that are summed to get the next term.

We define the $k$-generalized Fibonacci sequences by the following\footnote{Credit to Parks and Wills\cite{parks2022sums} for this design.}

\[ f_{n}^{k} = \begin{cases}\label{kf.definition} 
          0, & n<0, \\
          1, & n=0, \\
          \sum_{i=1}^{n} f_{n-i}^{(k)}, & n\geq1. \\
       \end{cases}
    \]

This paper is not singularly concerned with $k$-generalized Fibonaccis though. It also concerns the presence of Stirling numbers of the first kind. For $n>0$ and $k>0$, we define the unsigned Stirling numbers of the first kind by the following relation

\begin{equation}
    S(n,k) = (n-1)\cdot S(n-1,k) + S(n-1,k-1).
\end{equation}

While both numbers have recurrence relations, it may seem like these numbers are not related at all. However, when we sum the first $mk$ $k$-bonacci numbers, we can use a formula that involves the use of Stirling numbers, effectively connecting the Stirling numbers to these $k$-bonaccis.

\section{Preliminaries}

This section is split into 2 sections in order to completely describe my independent research into this topic.

\subsection{Initial Observations}

During a previous moment of research, I used Mathematica to generate solutions to a specific $mk$-sum of the $k$-bonaccis, in a general form as so, where each constant $c_{a,b}$ for $a,b\in\mathbb{Z}, a$ denoting the $a-1^{th}$ order polynomial, and $b$ denoting the $b^{th}$ term of the $a$ polynomial it belongs to.

\begin{equation}
    (c_1)2^{mk} + (c_{2,1}n +c_{2,2})2^{(m-1)k}+\dots+(c_{(m,1}n^{m}+\dots+c_{m,m})2^{k}
\end{equation}

This general formula came from the OEIS\cite{oeis}, where I originally looked at sequences with similar linear recurrences as a Fibonacci or 3-bonacci sequence. In this previous research project, my mentor and I settled on powers of 2 and polynomial associated with each power, as in the above form.

Listed below are the Mathematica generated formulas for the $2k$, $3k$, and $4k$ non-zero $k$-bonacci sums respectively.\footnote{When $m=1$, the sum is not very interesting; see Bravo and Luca's paper on $k$-generalized Fibonacci sequences from 2012.}

\begin{equation} \label{2kformula}
    (\frac{1}{2})4^k + (-\frac{1}{4}k+\frac{1}{4})2^k
\end{equation}

\begin{equation} \label{3kform}
  (\frac{1}{2})8^k + (-\frac{1}{2} k + \frac{1}{4})4^k + (\frac{1}{16}k^2 - \frac{3}{16}k +\frac{1}{8})2^k
\end{equation}

\begin{equation} \label{4kform}
    (\frac{1}{2})16^k+(k + \frac{1}{4})8^k + (\frac{1}{4}k^2-\frac{3}{8}k + \frac{1}{8})4^k + (-\frac{1}{96}k^3+\frac{1}{16}k^2-\frac{11}{96}k + \frac{1}{16})2^k 
\end{equation}

Now, to my independent research and calculations. The constants seem to be random, but ordered enough to have patterns. However, I tried writing and rewriting the solutions in order to find a pattern.

In order to organize the constants for each set of constants for a specified $m$-valued ``$mk$-bonacci", a ``$mk$-sum” pyramid, similar to Pascal's triangle is adopted. Pascal’s triangle is a triangular array which organizes the simplified binomial coefficients $n\choose k$ with each row dealing with a specific $n$-value. Each row contains $n$ values, where $n$ describes which row it is from the top of the pyramid. 

To define the constants that satisfy the formula of a sum for a given scalar multiple of $k$, we define the ``$mk$-sum” pyramid as a pyramid so that each row $m$, in numerical order, has $m$ terms. Defining the top row as row 1, each subsequent row of constants is row 2, 3, and so forth, until row $m$. Each row $m$ has the constants in standard form polynomial order of the $(m-1)^{th}$-order polynomial from the respective $mk$-sum formula.

To demonstrate, here is the set of constants for the 6$k$-sum formula.

\begin{figure}[ht]
    \centering{}
    \includegraphics[width=12.1cm]{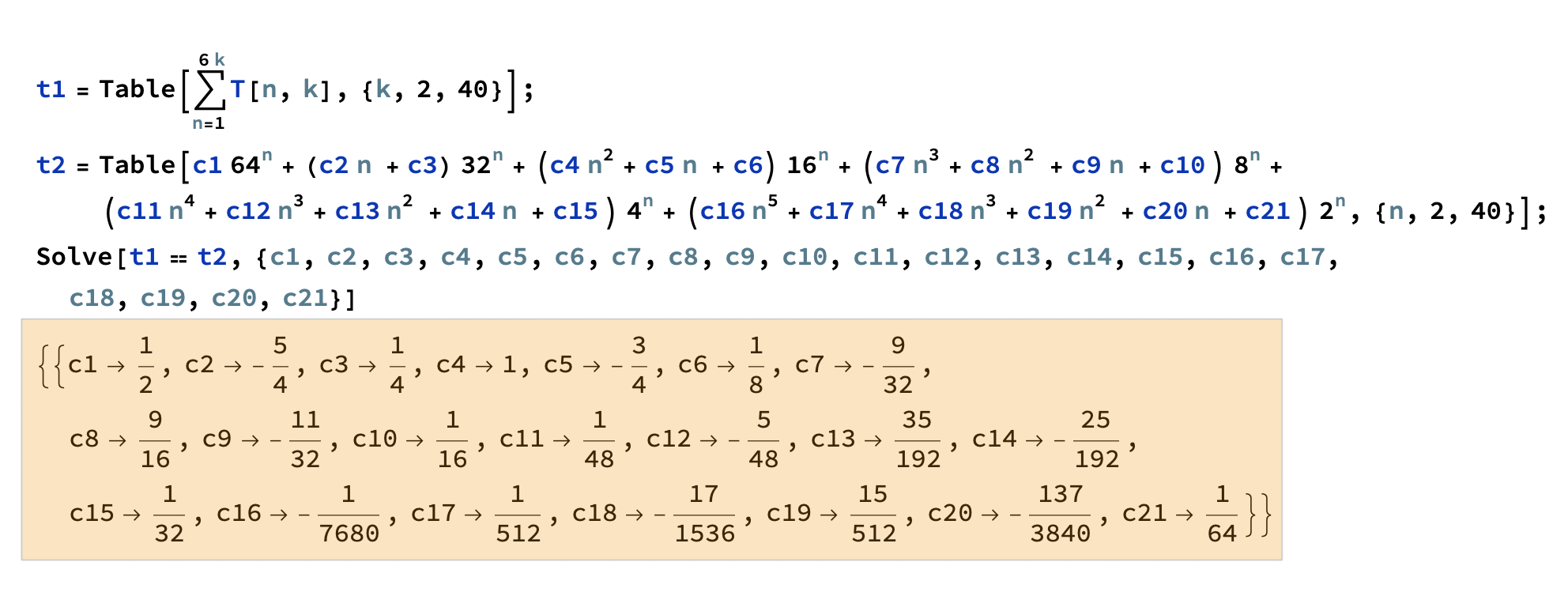}
    \label{6kpicture}
\end{figure}

Here is a pyramid for these constants. On the left for $r\geq1$, $r$ denotes the $r^{th}$ row of the pyramid, in which contain the constants for the respective $(r-1)^{th}$ order polynomial in the formula.

\begin{tabular}{rccccccccccccc}
$r_1$:&    &    &    &    &    &  $\frac{1}{2}$\\\noalign{\smallskip\smallskip}
$r_2$:&    &    &    &    &  $-\frac{5}{4}$ &    &  $\frac{1}{4}$\\\noalign{\smallskip\smallskip}
$r_3$:&    &    &    &  1 &    &  $-\frac{3}{4}$ &    &  $\frac{1}{8}$\\\noalign{\smallskip\smallskip}
$r_4$:&    &    &  $-\frac{9}{32}$ &    &  $\frac{9}{16}$ &    &  $-\frac{11}{32}$ &    &  $\frac{1}{16}$\\\noalign{\smallskip\smallskip}
$r_5$:&    &  $\frac{1}{48}$ &    &  $-\frac{5}{48}$ &    &  $\frac{35}{192}$ &    &  $-\frac{25}{192}$ &    &  $\frac{1}{32}$\\\noalign{\smallskip\smallskip}
$r_6$:&  $-\frac{1}{7680}$ &    &  $\frac{1}{512}$ &    &  $-\frac{17}{1536}$ &    &  $\frac{15}{512}$ &    &  $-\frac{137}{3840}$ &    &  $\frac{1}{64}$\\\noalign{\smallskip\smallskip}  
\end{tabular}

Unlike Pascal’s Triangle, it has been much more difficult to generalize an entry in the pyramid, due to the changing nature of this pyramid for each respective $k$-value. However, a few properties are identifiable through empirical observation.

\textbf{1.} Each entry in the first right diagonal will hold a value defined by $\frac{1}{2^m}$, where $m$ is the number for the row in which it is found.

Recall the general forms of the formulas for sums of 2$k$ integers, sums of 3$k$ integers, and both the 4$k$ and 5$k$ formulas.

\begin{center}
$(\frac{1}{2})4^n + (-\frac{n}{4}+\frac{1}{4})2^n$,for 2$k$ \\
$(\frac{1}{2})8^n + (-\frac{1}{2} n + \frac{1}{4})4^n + (\frac{1}{16}n^2 - \frac{3}{16}n +\frac{1}{8})2^n$,for 3$k$ \\
$(\frac{1}{2})16^n+(n + \frac{1}{4})8^n + (\frac{1}{4}n^2-\frac{3}{8}n + \frac{1}{8})4^n + (-\frac{1}{96}n^3+\frac{1}{16}n^2-\frac{11}{96}n + \frac{1}{16})2^n$,for 4$k$ \\  
\end{center}

There is a unique order polynomial attached to each major power of 2. The constant term, or final constant in the $r^{th}$ polynomial is always in the form $2^{k-r}$, where $r$ specifies the specific $r^{th}$ order polynomial of the formula that the given constant is found in, for each polynomial in the formula.

The common denominator for a given row $r$ is given in the form:
\begin{center}
\[(r-1)!\cdot2^{r} \] for $r\geq1$ 
\end{center}

\begin{proof}
    
Take the 6$k$ pyramid from above. Rewrite every row, with each row rewritten in its respective lowest common denominator.

\begin{figure}[ht]
    \centering{}
    \begin{tabular}{rccccccccccccc}
$r_1$:&    &    &    &    &    &  $\frac{1}{2}$\\\noalign{\smallskip\smallskip}
$r_2$:&    &    &    &    &  $-\frac{5}{4}$ &    &  $\frac{1}{4}$\\\noalign{\smallskip\smallskip}
$r_3$:&    &    &    &  $\frac{16}{16}$ &    &  $-\frac{12}{16}$ &    &  $\frac{2}{16}$\\\noalign{\smallskip\smallskip}
$r_4$:&    &    &  $-\frac{27}{96}$ &    &  $\frac{54}{96}$ &    &  $-\frac{33}{96}$ &    &  $\frac{6}{96}$\\\noalign{\smallskip\smallskip}
$r_5$:&    &  $\frac{16}{768}$ &    &  $-\frac{80}{768}$ &    &  $\frac{140}{768}$ &    &  $-\frac{100}{768}$ &    &  $\frac{24}{768}$\\\noalign{\smallskip\smallskip}
$r_6$:&  $-\frac{1}{7680}$ &    &  $\frac{15}{7680}$ &    &  $-\frac{85}{7680}$ &    &  $\frac{225}{7680}$ &    &  $-\frac{274}{7680}$ &    &  $\frac{120}{7680}$\\\noalign{\smallskip\smallskip}    
\end{tabular}
    \caption{The ``6$k$-sum" pyramid.}
    \label{fig:5}
\end{figure}

Form a sequence of each row's \textbf{lowest common denominator}, in order by row. Each denominator follows the following sequence.

\begin{center}
    $2,4,16,96,768,7680...$
\end{center}

Cross-referencing this with the OEIS\cite{oeis}, and this is sequence \href{https://oeis.org/A066318}{A066318}, where each term is in the form below.

\begin{center}
    $(r-1)!\cdot2^r$ for $r\geq1$
\end{center}
\end{proof}

Otherwise, these numbers fail to follow many other observable patterns. However, there is \textbf{one} observation that is quite remarkable in its uniqueness.

For a given $k$ value, write each entry in the $k^{th}$ row of the respective ``$k$-sum” pyramid for which it appears in lowest common denominator. To demonstrate, we will use row 4 of the 4$k$ pyramid and row 5 of the 5$k$ pyramid. The top sequence is row 4 and the bottom sequence is row 5 for the next few lines.
\begin{center}
    \[-\frac{1}{96}, \frac{1}{16}, -\frac{11}{96}, \frac{1}{16}\]
    \[\frac{1}{768}, -\frac{5}{384}, \frac{35}{768}, -\frac{25}{384}, \frac{1}{32}\]
\end{center}
Rewrite each set of constants in their respective lowest common denominator.
\[-\frac{1}{96}, \frac{6}{96}, -\frac{11}{96}, \frac{6}{96}\]
\[\frac{1}{768}, -\frac{10}{768}, \frac{35}{768}, -\frac{50}{768}, \frac{24}{768}\]

From here, take the \textbf{absolute value} of each term, and form a sequence utilizing the numerator of each fraction.

\[\frac{\textbf{1}}{96}, \frac{\textbf{6}}{96}, \frac{\textbf{11}}{96}, \frac{\textbf{6}}{96}\]
\[\frac{\textbf{1}}{768}, \frac{\textbf{10}}{768}, \frac{\textbf{35}}{768}, \frac{\textbf{50}}{768}, \frac{\textbf{24}}{768}\]

Each respective row above is transformed to the following.

\begin{center}
    \textbf{1, 6, 11,6} for row 4 of 4$k$ \\
    \textbf{1,10,35,50,24} for row 5 of 5$k$
\end{center}

And for each row from row $1$ to row $m$ of an ``$mk$-sum" pyramid, all stitched in order of increasing $m$, the following sequence is made. Note that the following numbers were derived from the constant $c_{a b}$ values that satisfy the general form, in order, from 2$k$ to 6$k$.

\[1, 1, 1, 1, 3, 2, 1, 6, 11, 6, 1, 10, 35, 50, 24, 1, 15, 85, 225, 274, 120\]

Cross-referencing the final sequences, it is shown that the numerators follow the pattern of the \emph{unsigned first-kind Stirling numbers}.\footnote{This sequence can be found on the OEIS\cite{oeis} at \href{https://oeis.org/A094638}{A094638}}.

\subsection{The Formula, Finalized}

I began to revisit this problem in spring 2024. In the form of jotting down a couple more $k$-bonacci pyramids, I noticed different new patterns, all resulting in the additions of term-specific multipliers, a Stirling polynomial, and other factors.

For example, I noticed that there was a pattern that I had completely missed, and reversed the order of each row in the "$mk$-pyramid", so that the Stirling coefficients were now in proper order. 

\begin{theorem}\label{t.main}
    For $k\geq$ 1 and $m\geq$ 1, the sum of the first $mk$ non-zero $k$-bonacci numbers is given by
    \begin{equation}\label{e.main}
        \sum_{i=1}^{m}\sum_{j=1}^{i}\frac{2^{(m-i+1)k}(-1)^{j-1}\cdot(k)^{j-1}\cdot S(i,j)\cdot(m-i+1)^{j-1}}{(i-1)!\cdot2^{i}}
    \end{equation}
    where $S(a,b)$ denotes the $(a,b)$ unsigned Stirling number of the first kind.
\end{theorem}

\section{Proof of Theorem 1}

We will prove this formula in two parts

We begin with the main equation from Parks and Wills\cite{parks2022sums}, namely the following

\begin{equation}\label{e.parks}
    \sum_{j=0}^{\lfloor\frac{n}{k+1}\rfloor} (-1)^{j} \binom{n-jk}{j} 2^{n-j(k+1)}.
\end{equation}

Since our Stirling formula for deals with sums of the first non-zero $mk$, we will replace the bounds of equation (\ref{e.parks}), substituting $mk-1$ for $n$,

\begin{equation}\label{e.parks.bounds1}
    \sum_{j=0}^{\lfloor\frac{mk-1}{k+1}\rfloor} (-1)^{j} \binom{(mk-1)-jk}{j} 2^{mk-1-j(k+1)}.
\end{equation}

Notice that the floor function in the summation reduces to $m-1$. We will also rewrite the power of 2 in the summation to split up the power.

\begin{equation}\label{e.parks.bounds2}
    \sum_{j=0}^{m-1} (-1)^{j} \binom{(mk-1)-jk}{i} 2^{(m-j)k} \cdot 2^{-j-1}.
\end{equation}

So this equation now looks fairly similar to the structure of our Stirling number formula. Notice that with our bounds, left "j" is always equal to right "i"$-1$. Trivially, it is clear that

\begin{equation}\label{trivial}
    \sum_{j=0}^{m-1} 2^{(m-j)k} = \sum_{i=1}^{m} 2^{(m-i+1)k}.
\end{equation}

Since these two are equal, the remaining expressions must be equal. 

\begin{equation}\label{e.finalbijection}
    \sum_{j=0}^{m-1} (-1)^{j} \binom{(mk-1)-jk}{j} 2^{-j-1} = \sum_{i=1}^{m} \sum_{j=1}^{i} \frac{(-1)^{j-1} S(i,j)\cdot(k)^{j-1}\cdot(m-i+1)^{j-1}}{(i-1)!\cdot 2^{i}}
\end{equation}

First, we would like to recall that the Stirling numbers of the first kind are not just numbers that follow a random recurrence relation. They also appears as the coefficients of $x$ in the falling factorial

\begin{equation}\label{fallingfactorial}
    (x)_{n} = x(x-1)(x-2)...(x-n+1)
\end{equation}

The binomial coefficient on the left side, since we focus on sums of any multiple of $k$, can be rewritten as so

\begin{equation}
    \frac{(((m-j)k-1)!}{j!\cdot((m-j)k-j)!}
\end{equation}

We notice that the numerator will simplify to the falling factorial of $((m-j)k-1)_{j}$. 

We can also call on the simple identity which relates the binomial coefficient to the falling factorial,

\begin{equation}
    \binom{n}{k} = \frac{(n)_{k}}{k!}
\end{equation}

With that, we will simplify the left side to 

\begin{equation}
    \sum_{j=0}^{m-1} (-1)^{j}\frac{((m-j)k-1)_{j}}{j!}\cdot2^{-j-1}
\end{equation}

Note that this is most nearly the falling factorial for $(mk)_n$, however with the beginning $mk$ missing. All that changes is that each power of $k$ in its expansion is simply to that power minus 1, what we can call an "near-falling factorial". Before going on, we introduce a final lemma that will enable us to complete this bijection.

\begin{lemma}\label{thelemma}
    For $n\geq1$ and $k\geq1$,
    \begin{equation}
        \sum_{j=0}^{m-1}(-1)^{j}\frac{(((m-j)k-1)_j)}{j!} = \sum_{i=1}^{m}\sum_{j=1}^{i}(-1)^{j-1}\frac{S(i,j)\cdot(k)^{j-1}\cdot(m-i)^{j}}{(i-1)!}
    \end{equation}
\end{lemma}
\begin{proof}
    Take the example of the following expansions of $((m-j)k-1)_3$ $(k-1)(k-2)$, $(2k-1)(2k-2)$, and $(3k-1)(3k-2)$. This polynomial corresponds to the Stirling numbers with $n=3$, with $S(3,1)=2$, $S(3,2)=-3$, and $S(3,3)=1$.
    \begin{center}
        \begin{tabular}{ r|c} 
            $m-j$ & expansion\\[1.5ex]
            \hline \\
            1 & $(k-1)(k-2) = 2-3k+k^{2}$ \\ [3.5ex]
            2 & $(2k-1)(2k-2) = 2-6k+4k^{2}$ \\ [3.5ex]
            3 & $(3k-1)(3k-2) = 2-9k+9k^{2}$ \\ [3.5ex]
        \end{tabular}
    \end{center}
    Notice how for each new multiple of $k$, each term simply multiplies by $m-j$ to the power that correspond to the power of $k$ in that term. To clarify notice the following,
    \begin{center}
        \begin{tabular}{ r|c} 
            $m-j$ & expansion\\[1.5ex]
            \hline \\
            1 & $(k-1)(k-2) = (1^{0})2-(1^{1})3k+(1^{2})k^{2}$ \\ [3.5ex]
            2 & $(2k-1)(2k-2) = (2^{0})2-(2^{1})3k+(2^{2})1k^{2}$ \\ [3.5ex]
            3 & $(3k-1)(3k-2) = (3^{0})2-(3^{1})3k+(3^{2})1k^{2}$ \\ [3.5ex]
        \end{tabular}
    \end{center}

    If we expand $(k-1)_n$, then by definition, we get

    \begin{equation}
        s(n,1)k^{0}+s(n,2)k^{1}+\dots+s(n,n)k^{n}.
    \end{equation}

    where $s(n,k)$ is the \textit{signed} Stirling number of the first kind.
    
    In this case, $m-j$ is 1. So whenever we expand $(k-1)_j$, we can generalize for $((m-j)k-1)_j$ by substituting in $(m-j)k$ for any term that includes a $k$, hence the expansion becomes
    
    \begin{equation}
        (m-j)^{0}S(j,1)k^{0}+(m-j)^{1}S(j,2)k^{1}+\dots+(m-j)^{n}S(j,n)k^{n-1}.
    \end{equation}

    The $n$ variable we use is simply a placeholder that increments each power of $k$ and increments through all Stirling numbers in the form $S(j,n)$ where $1\leq n \leq j$. However, to match our Stirling formula, we change the $j$ variable to an $i$,  turn the $n$ placeholder into a $j$, and sum all values of $j$ from 1 to $i$

    \begin{equation}\label{finaleq}
        \sum_{i=0}^{m-1}(-1)^{i}\sum_{j=0}^{i}\frac{(m-i)^{j}s(i,j)\cdot(k)^{j-1}}{j!} = \sum_{i=1}^{m}\sum_{j=1}^{i}\frac{(-1)^{j-1}S(i,j)\cdot(k)^{j-1}\cdot(m-i)^{j}}{(i-1)!}
    \end{equation}

    The bounds may differ, but they cycle through the same set of numbers, just as in \ref{trivial} above. Notice also that both alternating $-1$s on either side of the equation achieve the same purposes. The left hand side simply distributes the $-1$ across the entire "near-falling factorial" whereas the right side applies the power of $-1$ to each term of the expanded "near-falling factorial" directly, instead of first signing the Stirling number and applying the $-1$ to the entire polynomial, alternating for each successive polynomial.
    
    With that, we complete the proof of \ref{thelemma}, the other half of the formula.
\end{proof}

Combining equations \ref{trivial} and \ref{thelemma}, the proof is complete, and we have proven that our experimentally determined formula is a deconstruction of a more-closed binomial formula. Recall \ref{t.main} and \ref{e.parks.bounds2}, and know that $2^{-j-1}$ and $\frac{1}{2^i}$, are the same with the different bounds.

\subsection{The implications of this}

Before writing this paper, I did not realize that this problem would have a simpler answer than I expected. However, the existence of Stirling numbers appearing in the binomial coefficients implies that Stirling numbers have a larger role in combinatorial formulas.

It does open the question of whether there is a proper combinatorial interpretation of my Stirling formula. However, that is beyond the scope of this paper. While the factorial has a simple combinatorial argument, the Stirling polynomial's combinatorial interpretation would definitely be harder to come to.

\section{Acknowledgements}

This paper functions as an independent continuation of the previous paper I wrote for Pioneer Academics with the help of Professor Greg Dresden of Washington and Lee University. Before anything, I would like to acknowledge his help in the earlier paper, and his current encouragement in proving the resulting summation formula.

\nocite{*} 
\bibliographystyle{plain}
\bibliography{references}

\end{document}